\documentclass[a4paper,11pt]{amsart}
\usepackage{hyperref,latexsym}
\usepackage{enumerate}

\theoremstyle{plain}
\newtheorem{theorem}{Theorem}[section]
\newtheorem{lemma}[theorem]{Lemma}

\theoremstyle{definition}
\newtheorem{definition}[theorem]{Definition}

\theoremstyle{remark}

\newcommand{\abs}[1]{\left|#1\right|}

\begin{document}

\title[p-system of non-linear elasticity]
      {From polynomial integrals of Hamiltonian flows to a model of
non-linear elasticity}

\date{December 2012}
\author{Misha Bialy and Andrey E. Mironov}
\address{M. Bialy, School of Mathematical Sciences, Raymond and
Beverly Sackler Faculty of Exact Sciences, Tel Aviv University,
Israel} \email{bialy@post.tau.ac.il}
\address{A.E. Mironov, Sobolev Institute of Mathematics and
Laboratory of Geometric Methods in Mathematical Physics, Moscow
State University } \email{mironov@math.nsc.ru}
\thanks{M.B. was supported in part by ISF grant 128/10 and A.M. was
supported by the Presidium of the Russian Academy of Sciences (under
the program "Nonliner Systems in Geometry"); grant MD-5134.2012.1
from the President of Russia; and a grant from Dmitri Zimin's
"Dynasty" foundation.
 It is our pleasure to thank these
funds for the support}

%\subjclass[2000]{}

\begin{abstract}
We prove non existence of smooth solutions of a quasi-linear system
suggested by Ericksen in a model of Nonlinear Elasticity. This
system is of mixed elliptic-hyperbolic type. We discuss also a
relation of such a system to polynomial integrals of Classical
Hamiltonian systems.
\end{abstract}

\maketitle

%%%%%%%%%%%%%%%%%%%%%%%%%%%%%%%%%%%%%%%%%%%%%%%%%%%%%%%%%%%%%%%%%%%%%%%%%%
\section{Introduction and main results}
In this paper we study smooth periodic solutions of the following
equation
\begin{equation}\label{1}
 u_{tt}+(\sigma(u))_{xx}=0.
\end{equation}
Throughout this paper the function $u(t,x)$ is assumed to be
periodic in $x$, $u(t,x+1)=u(t,x)$ and $C^2-$smooth on the whole
cylinder or on the half-cylinder. Equation (\ref{1}) is a
compatibility condition of the quasi-linear $2\times 2$ system
usually called $p$-system (here we shall use $\sigma$ instead of $p$
since $p$ is reserved for momentum which appears below).

\begin{equation}
 \left\{
\begin{array}{l}\label{2}
 u_t=-v_x\\
 v_t=(\sigma(u))_x.\\
\end{array} \right.
\end{equation}

One can easily see that the periodicity of $u$ leads to
\begin{equation}
\label{C}
 v(t,x+1)=v(t,x)+C,
\end{equation}
$C$ is a constant. Therefore $v$ is a sum of a linear function and a
function periodic in $x$.

The system (\ref{2}) and the equation (\ref{1}) appear in many
applications. The purpose of this paper is to prove that there are
no smooth periodic solution of these equations. It is widely known,
starting from \cite{L}, that smooth solutions for Hyperbolic
quasi-linear system \textit{generically} do not exist after a finite
time. However rigorous proof of this general belief are not
immediate and usually is not that simple. It is important that our
conditions on $\sigma$ are such that system (\ref{2}) is of mixed
elliptic-hyperbolic type and therefore the analysis of elliptic and
hyperbolic zones as well as the boundary between them is required
(we refer to \cite{T} for a recent survey on mixed problems).

In this paper we shall discuss two appearances of the equation
(\ref{1}) (\ref{2}): the first is related to Classical mechanics and
the second is a model suggested by Ericksen from Nonlinear
Elasticity.

In Classical mechanics one is looking for conserved quantities or
integrals of Hamiltonian flow
\begin{equation}
\label{H}
 \frac{dx}{dt}=\frac{\partial H}{\partial p},\quad \frac{dp}{dt}=-\frac{\partial H}{\partial x}.
\end{equation}
In this context, assume $H=\frac{1}{2}p^2+u(t,x)$ be a Hamiltonian
function with potential function $u(t,x)$. Write
$F=\frac{1}{3}p^3+up+v$. Then the condition that $F$ has constant
values along the flow lines of $H$ reads:
$$
 \frac{dF}{dt}=F_t+pF_x-u_xF_p=0
$$
and leads to the system (2) with the function
$\sigma(u)=\frac{u^2}{2}$. In this case (\ref{1}) turns out to be
dispersionless Boussinesq equation. This fact was first noticed by
V.V. Kozlov in \cite{Kozlov} where trigonometric polynomial
solutions were considered. Doubly periodic solution were further
studied in \cite{B1},\cite{B2}(other quasi-linear systems arising in
Classical mechanics are discussed in \cite{BM1},\cite{BM2}). Our
method in this paper is a continuation of \cite{B0} where the case
of quadratic-like function $\sigma$ is studied.

Another appearance of (\ref{1}) and (\ref{2}) comes from a model of
non-linear elasticity suggested by Ericksen \cite{E}. In this model
$\sigma$ appears to be a cubic-like function (of type II below). We
refer to the paper \cite{S-P}, where this model is discussed.

Motivated by these two applications we shall consider two types of
$\sigma$: \vskip5mm
 I. Quadratic-like type.

The function $\sigma$ is strictly convex function having a minimum
(see Fig. 1).

\vskip50mm

\begin{picture}(170,100)(-100,-80)
\put(60,-15){\vector(0,4){130}}
%\put(80,-30){\line(0,4){125}}
%\put(-15,25){\vector(3,0){185}}
\put(15,25){\vector(3,0){185}}

\qbezier(80,100)(100,5)(160,100)
%\qbezier(120,60)(140,100)(160,-5)

%\put(105,25){\circle*{2}}
%\put(132,25){\circle*{2}}
%\put(102,10){\shortstack{$\alpha$}}
%\put(127,10){\shortstack{$\beta$}}

\put(195,30){\shortstack{$u$}}
\put(65,112){\shortstack{$\sigma(u)$}}
\put(50,-35){\shortstack{Fig. 1}}
\end{picture}
\newpage
II. Cubic-like type. \vskip4mm

The function $\sigma$ behaves like a cubic polynomial in $u$ (see
Fig. 2), i.e.
$$
 \sigma'(u)<0, \quad for \quad
 u<\alpha\quad and \quad u>\beta,
$$
$$
 \sigma'(u)> 0, \quad for \quad
 u\in(\alpha,\beta),$$
$$
 \sigma''(u)>0, \quad for \quad  u\leq\alpha \quad and \quad \sigma''<0, \quad for \quad u\geq\beta.
$$

\vskip40mm

\begin{picture}(170,100)(-100,-80)
\put(60,-15){\vector(0,4){130}}
%\put(80,-30){\line(0,4){125}}
%\put(-15,25){\vector(3,0){185}}
\put(15,25){\vector(3,0){185}}

\qbezier(80,100)(100,5)(120,60)
\qbezier(120,60)(140,100)(160,-5)

\put(105,25){\circle*{2}}
\put(132,25){\circle*{2}}
\put(102,10){\shortstack{$\alpha$}}
\put(127,10){\shortstack{$\beta$}}

\put(195,30){\shortstack{$u$}}
\put(65,112){\shortstack{$\sigma(u)$}}
\put(50,-35){\shortstack{Fig. 2}}
\end{picture}

Our main result is given in the following two theorems where the
case of quadratic like function $\sigma$ was proved in \cite{B0} and
is included here for completeness.

\begin{theorem}
\label {main1} {\it Let $\sigma(u)$ be of type I,II. Then any
$C^2$-solution of (2) defined on the half-cylinder
$[t_0,+\infty)\times \mathbf{S}^1$:
$$
 u(t,x+1)=u(t,x),\ v(t,x+1)=v(t,x), t\geq t_0,
$$
which has initial values in the Hyperbolic region
$U_h=\{u<\alpha\}\cup\{u>\beta\}$ must be constant.}
\end{theorem}
To get the result on the whole infinite cylinder one can remove the
initial Hyperbolicity assumption:

\begin{theorem}
\label{main2}
 {\it If the function $\sigma$ is of type I or type II, then any
$C^2$-solution $(u(t,x),v(t,x))$ of the system (\ref{2}) defined on
the whole cylinder $\mathbf{R}\times \mathbf{S}^1$ so that,
$$
 u(t,x+1)=u(t,x),\ v(t,x+1)=v(t,x)
$$ must be constant.}
\end{theorem}

Let us remark that in these two theorems $v$ is assumed to be
periodic together with $u$, that is the constant $C$ in formula
(\ref{C}) is assumed to be zero. Moreover, the pair $u=-Ct, v=Cx$ is
obviously a solution of (\ref{2}) for any $C$. In this example $u$
is obviously periodic in $x$, but $v$ is not. It is an open question
if there are other smooth global solutions of (\ref{2}) periodic in
$x$ for $C\ne 0$. Our next result says that there are no, if $u$ is
assumed to be periodic in both $t$ and $x$.

\begin{theorem}
\label{integral}
 {\it If $\sigma=\frac{u^2}{2}$  then any $C^2$-solution
$(u(t,x),v(t,x))$ of (\ref{2}) with $u$ being doubly periodic
$$
 u(t+1,x)=u(t,x+1)=u(t,x)
$$
must be constant.}
\end{theorem}

Remarkably in the proof below one shows by means of dynamical
systems theory that the function $v$ must be periodic either. This
fact gives a reduction of Theorem \ref{integral} to
Theorem \ref{main2}.

As a simple corollary of the analysis given in the proofs of
Theorems \ref{main1}, \ref{main2} we can strengthen the result of
Theorem \ref{integral} as follows:

\begin{theorem}
\label{bounded}
 {\it Let $(u(t,x),v(t,x))$ be a $C^2$-solution of (2)
with $\sigma(u)$ being either of type I or type II, then if in
addition $u(t,x)$ is bounded and periodic
$$
 u(t,x+1)=u(t,x)
$$
 then $(u,v)$ must be constant. }
\end{theorem}

Several remarks and questions are in order:

1. It is still not clear to us how to classify all solutions of
(\ref{1}) periodic in $x$ with no assumptions on periodicity of $v$
or boundedness of $u$.

2. It was crucial for the proof of our main theorems that within the
Hyperbolic regions the eigenvalues are genuinely non-linear. It
would be interesting to understand the case when the function
$\sigma$ behaves like in van der Waals model (see \cite{Slemrod})
where the genuine non-linearity condition is violated.

3. Another classical tool \cite {CF} for the system (\ref{2}) which
could be used near the points where the mapping
$(t,x)\rightarrow(u,v)$ is local diffeomorphism is the Hodograph
method. However, we don't know how this method can be used globally,
taking care on the singularities, in order to get another proof of
our results. Let us also remark that it would be very interesting to
find a connection with a recent theory of normal forms of the
singularities and the so called Universality conjecture developed in
\cite{Du}.

The paper is organized as follows. In the next section we show the
Dynamical systems argument reducing Theorem \ref{integral} to
Theorem \ref{main2}. Then in Section 4 we prove two key Lemmas. In
Section 5 we treat Hyperbolic part of the Theorems \ref{main1},
\ref{main2} and get Theorem \ref{bounded} as a corollary. Section 6
provides a convexity argument for the Elliptic zones.

%%%%%%%%%%%%%%%%%%%%%%%%%%%%%%%%%%%%%%%%%%%%

\section{Reduction of Theorem \ref{integral} to Theorem \ref{main2}}
In this section we give a reduction of Theorem \ref{integral} to
Theorem \ref{main2} based on Dynamical Systems ideas.
\begin{proof}
In this theorem $\sigma(u)=\frac{u^2}{2}.$ Let $u(t,x)$ be a
$C^2$-solution of (\ref{1}) which is periodic both in $t$ and $x$.
Then let $v(t,x)$ be a function such that the pair $(u,v)$ solves
(\ref{2}). The function $v$ is defined up to a constant and is not
necessarily periodic. It can be written
$$
 v(t,x)=At+Bx+{\tilde v}(t,x),
$$
where $A,B$ are some constants and ${\tilde v}$ is periodic in both
$t,x$. As we mentioned above system (\ref{2}) is equivalent to the
fact that the function $$F=\frac{1}{3}p^3+u(t,x)p+v(t,x)$$ has
constant values along the Hamiltonian flow of
$$H=\frac{1}{2}p^2+u(t,x).$$ It is a standard result of calculus
variations, that for any positive integer $m$ and any integer $n$
there exists periodic orbit of the Hamiltonian flow of type $(m,n)$.
This means that there is a solution of the Hamilton equations
(\ref{H}) $(x(t),p(t))$ with the property

\begin{equation}\label{3}
 x(t+m)=x(t)+n,\ p(t+m)=p(t),\ m,n\in \mathbb{Z},\ m>0.
\end{equation}
Along this orbit $F$ has a constant value. Therefore
\begin{equation}
\label{F}
 F(t+m,x(t+m),p(t+m))=F(t,x(t),p(t)).
\end{equation}
Left hand side of (\ref{F}) with the help of (\ref{3}) and
periodicity of $u$ equals
$$
 F(t+m,x(t+m),p(t+m))$$
$$=\frac{1}{3}p^3(t)+u(x(t)+n,t+m)p(t)+v(x(t)+n,t+m)$$
$$
 =\frac{1}{3}p^3(t)+u(x(t),t)p(t)+v(x(t)+n,t+m)$$
$$
 =\frac{1}{3}p^3(t)+u(x(t),t)p(t)+Am+Bn+At+Bx(t)+{\tilde v}(x(t),t).
$$
The right hand side of (\ref{F}) is the following
$$
 F(t,x(t),p(t))=\frac{1}{3}p^3(t)+u(x(t),t)p(t)+At+Bx(t)+{\tilde v}(x(t),t).
$$
Equating the expressions of the right and the left hand side we get
the identity:
$$
 Am+Bn=0
$$
for any $m,n$. So $A=B=0$. Thus $v$ is a periodic function so
Theorem \ref{main2} applies and yields the result.

\end{proof}

%%%%%%%%%%%%%%%%%%%%%%%%%%%%%%%%%%%%%%%%%%%%

\section{Preliminaries on the Hyperbolic regions}
 Here we shall collect the needed facts on the p-systems (see \cite{Serre} and also \cite{Slemrod}).
The system (2) can be written in the form
$$
  \left(
  \begin{array}{c}
   u \\
  v\\
  \end{array}\right)_t+A(u,v)
  \left(
  \begin{array}{c}
   u \\
  v\\
  \end{array}\right)_x=0,\ \
  A  \left(
  \begin{array}{cc}
   0 & 1 \\
   -\sigma'(u) & 0\\
  \end{array}\right).
$$
Let $(u(t,x),v(t,x))$ be a solution periodic in $x$.
Denote by
$$
 U_e=\{(t,x): u(t,x)\in(\alpha,\beta)\}
$$
the elliptic region where the matrix $A$ has complex eigenvalues.
Denote by $U_h$ the hyperbolic region consisting of two disjoint
domains $$U_h=U_\alpha\cup U_\beta,$$ where
$$
U_{\alpha}=\{(t,x): u(t,x)<\alpha\}, \ U_{\beta}=\{(t,x):
u(t,x)>\beta\}.
$$
On $U_h$ the matrix $A$ has two real distinct eigenvalues
$$
 \lambda_1=\sqrt{-\sigma'(u)},\ \lambda_2=-\sqrt{-\sigma'(u)}.
$$
The boundaries of $U_e$ and $U_h$ belong to
$$
 U_0=\{(t,x): u(t,x)=\alpha \quad or \quad u(t,x)=\beta\}.
$$

It is important for the sequel that $U_\alpha, U_\beta$ have
disjoint closure on the cylinder. Therefore the characteristics of
the system cannot jump from one of the domains $U_\alpha$ or
$U_\beta$ to the other. We analyze the behavior of the
characteristics in $U_\alpha$, and for the $U_\beta$ all conclusions
are analogous with the obvious changes.

In $U_{\alpha}$ there are Riemann invariants
$$
 r_{1}=v- \int_u^\alpha \sqrt {-\sigma'(u)}du,\quad  r_{2}=v+ \int_u^\alpha \sqrt {-\sigma'(u)}du,
$$
$$
 (r_i)_t+\lambda_i(r_i)_x=0, \ i=1,2,
$$
 and so $u,v$ can be recovered
from the Riemann invariants by the formulas:
\begin{equation}\label{q}
 v=\frac
{r_1+r_2}{2},\quad u=q^{-1}\left(\frac{r_2-r_1}{2}\right),
\end{equation}
where by definition
$$
 q(u):=\int_u^\alpha \sqrt {-\sigma'(s)}ds,
$$
 is a positive monotone decreasing function for $u<\alpha$
with
$$
 q(\alpha)=0,\quad  q'(u)=-\sqrt{-\sigma'(u)},\quad
 q''(u)=\frac{\sigma''(u)}{2\sqrt{-\sigma'(u)}}.
$$
It is crucial fact
that both eigenvalues are genuinely non-linear in $U_{\alpha}$ by the formulas:
$$
 (\lambda_1)_{r_1}=(\lambda_2)_{r_2}=\frac{\sigma''(u)}{4\sigma'(u)}\ne 0.
$$
Notice that near the boundary $\partial U_{\alpha}$ the non-linearity becomes
infinite. Moreover, verifying literarily the Lax method \cite{L} for
our $p$-system one arrives to the following Riccati equations along
characteristics of the first and the second eigenvalues:
\begin{equation}\label{R}
L_{v_1}(z_1)+kz_1^2=0,\ \ L_{v_2}(z_2)+kz_2^2=0
\end{equation}
where
$$
 z_1:=(r_1)_x (-\sigma'(u))^{\frac{1}{4}}, \quad z_2:=(r_2)_x
 (-\sigma'(u))^{\frac{1}{4}}, \quad k:=-\frac
 {\sigma''(u)}{4(-\sigma'(u))^{\frac{5}{4}}},
$$
and
$$
 L_{v_1}=\partial_t+\lambda_1\partial_x,\ \ L_{v_2}=\partial_t+\lambda_2\partial_x
$$
stand for derivatives along the first and the second characteristic
fields respectively.

\section{The Key Lemmas} The following two lemmas are crucial for
the proofs. Recall again that we shall state them for $U_\alpha$ so
that for $U_\beta$ the statements go with obvious replacements.

\begin{lemma}
\label{l1}
 Along the characteristic curves we have:

1)\ If a characteristic curve of the first or of the second
eigenvalue starting from the initial time $t_0$ reaches the boundary
$\partial U_\alpha$ in a finite time $t_+>t_0$ (respectively
$t_-<t_0$), then the corresponding Riemann invariant satisfies
$r_x\leq 0$ (resp. $r_x\geq 0$) along this characteristic.

2)\ If a characteristic curve of the first or of the second
eigenvalue extends to a semi-infinite interval $( [t_0,+\infty)$
(resp. $(-\infty,t_0]$), then  for the corresponding Riemann
invariant either  $(r)_x \leq 0$ (resp. $r_x\geq 0$) or $-u$ and
$-\sigma'(u)$ tend to $ +\infty$ along this characteristic curve
when $t \rightarrow +\infty$ (resp. $t \rightarrow -\infty$).

\end{lemma}
\begin{proof}
To prove the Lemma we use the exact formula for the solutions of
equation (\ref{R}):
$$
 z(t)=\frac{z(t_0)}{1+z(t_0)\int_{t_0}^t k(s)ds}.
$$
Let us prove the first part of the Lemma. Suppose that
characteristic extends to the maximal interval $[t_0;t_+).$ Recall
that the characteristics are solutions of the equation
\begin{equation}
\label{ode} \dot{x}=\pm\sqrt{-\sigma'(u)}.
\end{equation}
It is a standard fact in ODE theory, that if a characteristic curve
approaches the boundary of $U_{\alpha}$, so that $u$ tends to
$\alpha$, then the characteristic curve must converge to a limit
point say $(t_+,x_+)$ on the boundary $\partial U_\alpha$ (see Fig.
3).

\vskip20mm

\begin{picture}(170,100)(-100,-80)
\put(-35,-25){\vector(0,4){80}}
\put(-35,-25){\vector(4,0){220}}

\qbezier(-20,-10)(50,-5)(100,20)

\qbezier(100,20)(100,40)(140,41)
\qbezier(140,41)(180,40)(180,20)
\qbezier(180,20)(180,0)(140,-1)
\qbezier(140,-1)(100,0)(100,20)

\qbezier(105,30)(105,20)(105,10)
\qbezier(115,35)(115,20)(115,5)
\qbezier(125,37)(125,20)(125,3)
\qbezier(135,39)(135,20)(135,1)
\qbezier(145,39)(145,20)(145,1)
\qbezier(155,37)(155,20)(155,3)
\qbezier(165,35)(165,20)(165,5)
\qbezier(175,30)(175,20)(175,10)

\put(-25,0){\shortstack{$\gamma$}}
\put(145,44){\shortstack{$\partial U_{\alpha}$}}
\put(12,35){\shortstack{$U_{\alpha}$}}
\put(55,27){\shortstack{$(t_+,x_+)$}}
\put(-5,4){\shortstack{$(t_0,x(t_0))$}}
\put(188,-28){\shortstack{$t$}}
\put(-37,58){\shortstack{$x$}}

\put(100,20){\circle*{3}}
\put(15,-6){\circle*{3}}

\put(60,-50){\shortstack{Fig. 3}}

\end{picture}

Moreover, it follows then that the integral
\begin{equation} \int_{t_0}^{t_+}
 k(s)ds=-\int_{t_0}^{t_+}\frac{\sigma''(u(s,x(s)))}{4(-\sigma'(u(s,x(s)))^{\frac{5}{4}}}ds
 \label{int}
\end{equation}
 diverges to $-\infty$. Indeed, for
$t\rightarrow t_+$ the function $u(t,x(t))\rightarrow \alpha$  and
can be estimated from above by
$$
 \abs{u(t, x(t))-\alpha}\leq C_1 \abs
 {t-t_+} ,
$$
also for $u$ close to $\alpha$ one can estimate:
$$
 \abs{\sigma'(u)}=\abs{\int_u^\alpha \sigma''(u)du}\leq
 C_2\abs{u-\alpha}, \ where\
 C_2=\max_{u\in[\alpha-1,\alpha]}\sigma''(u).
$$

So the nominator in of the integrand of (\ref{int}) is bounded away
from zero and the denominator is less or equal then
$
 C_1C_2 \abs {t-t_+ }^{\frac{5}{4}},
$
and
$$
 -\int_{t_0}^{t_+}\frac{\sigma''(u)}{4(-\sigma'(u))^{\frac{5}{4}}}ds<
 -\int_{t_0}^{t_+}\frac{C_0}{C_1C_2 \abs {t-t_+ }^{\frac{5}{4}}}ds\rightarrow -\infty,
$$
 thus the integral (\ref{int}) diverges. This proves the
first part of the lemma.

The second part of the lemma is as follows.

\vskip2.5cm

\begin{picture}(170,100)(-100,-80)
\put(-35,-25){\vector(0,4){100}}
\put(-35,-25){\vector(4,0){220}}

%\qbezier(-20,30)(75,35)(110,70)
\qbezier(-20,30)(75,35)(120,75)
\qbezier(-20,10)(75,15)(150,60)
\qbezier(-20,-10)(95,-5)(180,45)

\qbezier(-20,32)(-20,50)(-20,75)
\qbezier(-20,-12)(-20,-20)(-20,-23)
\qbezier(-10,33)(-10,50)(-10,75)
\qbezier(-10,-11)(-10,-20)(-10,-23)
\qbezier(0,34)(0,50)(0,75)
\qbezier(0,-10)(0,-20)(0,-23)
\qbezier(10,35)(10,50)(10,75)
\qbezier(10,-9)(10,-20)(10,-23)
\qbezier(20,36)(20,50)(20,75)
\qbezier(20,-8)(20,-20)(20,-23)
\qbezier(30,37)(30,50)(30,75)
\qbezier(30,-7)(30,-20)(30,-23)
\qbezier(40,39)(40,50)(40,75)
\qbezier(40,-6)(40,-20)(40,-23)
\qbezier(50,41)(50,50)(50,75)
\qbezier(50,-5)(50,-20)(50,-23)
\qbezier(60,44)(60,50)(60,75)
\qbezier(60,-3)(60,-20)(60,-23)
\qbezier(70,48)(70,50)(70,75)
\qbezier(70,-1)(70,-20)(70,-23)
\qbezier(80,53)(80,54)(80,75)
\qbezier(80,2)(80,-20)(80,-23)
\qbezier(90,58)(90,60)(90,62)
\qbezier(90,5)(90,-20)(90,-23)
\qbezier(100,71)(100,74)(100,75)
\qbezier(100,7)(100,-20)(100,-23)
\qbezier(110,69)(110,74)(110,75)
\qbezier(110,10)(110,-20)(110,-23)

\qbezier(120,14)(120,-20)(120,-23)
\qbezier(130,18)(130,-20)(130,-23)
\qbezier(140,22)(140,-20)(140,-23)
\qbezier(150,27)(150,-20)(150,-23)
\qbezier(160,32)(160,-20)(160,-23)
\qbezier(170,21)(170,-20)(170,-23)

\qbezier(180,21)(180,-20)(180,-23)
\qbezier(180,43)(180,37)(180,36)

\put(-30,10){\shortstack{$\gamma$}}
\put(165,26){\shortstack{$\partial U_{\alpha}$}}
\put(82,66){\shortstack{$\partial U_{\alpha}$}}
\put(-15,0){\shortstack{$(t_0,x(t_0))$}}
\put(188,-28){\shortstack{$t$}}
\put(-37,78){\shortstack{$x$}}

\put(4,12){\circle*{3}}
%\put(15,5){\circle*{3}}

\put(60,-50){\shortstack{Fig. 4}}

\end{picture}

For an infinite characteristic (Fig. 4)
$$
 \gamma_1(t)=(t, x_1(t)), \ t\in[t_0;t_+),\ t_+=+\infty,
$$
there are two possibilities.

The first is when the integral (\ref{int}) diverges to $-\infty$, in
this case $r_x\leq 0$ exactly as in the previous case.

\vskip2.5cm

\begin{picture}(170,100)(-100,-80)
\put(-35,-25){\vector(0,4){100}}
\put(-35,-25){\vector(4,0){220}}
\qbezier(-20,10)(75,15)(150,60)
\qbezier(-20,60)(80,56)(150,10)

%\put(0,-17){\shortstack{$\gamma_1$}}
%\put(104,-15){\shortstack{$\gamma_2$}}
%\put(149,64){\shortstack{$\partial U_{\alpha}$}}
%\put(75,53){\shortstack{$U_{\alpha}$}}

\put(115,33){\shortstack{$P\in A$}}
\put(-16,68){\shortstack{$(t_0,x_2(t_0))$}}

\put(188,-28){\shortstack{$t$}}
\put(-37,78){\shortstack{$x$}}
\put(-33,10){\shortstack{$\gamma_1$}}
\put(-33,60){\shortstack{$\gamma_2$}}
\put(102,35){\circle*{3}}
\put(5,58){\circle*{3}}

\put(60,-50){\shortstack{Fig. 5}}

\end{picture}

In the second possibility the integral (\ref{int}) is converging. In
this case we need to prove that $-u$ and $-\sigma'(u)$ must tend to
$+\infty$ as $t\rightarrow +\infty$. For this we use the periodicity
of $u,v$ in $x$. Let $P=(t_*, x_1(t_*))$ be any point on the
characteristic $\gamma_1$. Consider characteristic of the second
eigenvalue $\gamma_2=(t, x_2(t))$ passing through this point $P$.
Let us follow $\gamma_2$ backwards. Either the characteristic
$\gamma_2$ can be extended backwards on $[t_0;t_*)$ (in such a case
we shall call point $P$ accessible from $t_0$, see Fig. 5) or it can
be extended backward to the maximum interval of existence $(t_-;
t_*), t_0<t_-$ (then $P$ will be called unaccessible).

Notice that in the last case the possibility $x_2(t)\rightarrow
+\infty, t\searrow t_-$ is easily excluded because by periodicity in
$x$, the solutions of the equation (\ref{ode}) are bounded on any
compact interval of time. Thus in case point $P$ is unaccessible the
characteristic $\gamma_2$ reaches the boundary $\partial U_\alpha$
in a backward time (Fig. 6).

Denote by $\textsf{A}$ and $\textsf{B}$ the set of all accessible
and unaccessible points respectively on the characteristic
$\gamma_1$. The set $\textsf{A}$ is obviously open in $\gamma_1$ and
so consists of the union of open intervals. From the periodicity
condition it follows that $r_2$ is bounded on the line $t=t_0$.
Moreover, since $r_2$ is preserved along characteristics of the
second family and by the formula (\ref{q}), it follows that $u$ is
uniformly bounded on all intervals of the set $\textsf{A}$ (see Fig.
6).

\vskip2.5cm

\begin{picture}(170,100)(-100,-80)
\put(-35,-25){\vector(0,4){100}}
\put(-35,-25){\vector(4,0){220}}

\qbezier(-15,-30)(-15,50)(-15,75)
%\qbezier(5,-30)(5,50)(5,75)
\qbezier(-15,-20)(150,-10)(180,75)

\qbezier(15,10)(15,25)(35,25)
\qbezier(35,25)(55,25)(55,10)
\qbezier(55,10)(55,-5)(35,-5)
\qbezier(35,-5)(15,-5)(15,10)

\qbezier(50,0)(50,-2)(50,20)
\qbezier(40,-4)(40,10)(40,23)
\qbezier(35,-4)(35,10)(35,24)
\qbezier(30,-3)(30,10)(30,23)
\qbezier(25,-2)(25,10)(25,22)

\qbezier(20,0)(20,10)(20,20)
\qbezier(45,-3)(45,10)(45,23)

\qbezier(75,50)(75,65)(95,65)
\qbezier(95,65)(115,65)(115,50)
\qbezier(115,50)(115,35)(95,35)
\qbezier(95,35)(75,35)(75,50)

\qbezier(110,61)(110,50)(110,39)
\qbezier(100,63)(100,50)(100,37)
\qbezier(95,64)(95,50)(95,36)

\qbezier(90,63)(90,50)(90,37)
\qbezier(85,63)(85,50)(85,37)

\qbezier(80,61)(80,50)(80,39)
\qbezier(105,62)(105,50)(105,38)

\qbezier(-15,73)(110,68)(170,55)
\qbezier(-15,47)(100,36)(142,25)
\qbezier(115,46)(135,44)(158,38)

\qbezier(-15,28)(80,25)(108,5)
\qbezier(-15,37)(95,32)(125,13)
\qbezier(55,7)(70,6)(85,-3)
\qbezier(-15,-2)(25,-3)(60,-10)

\put(-6,13){\shortstack{$\partial U_{\alpha}$}}
\put(54,48){\shortstack{$\partial U_{\alpha}$}}

\put(188,-28){\shortstack{$t$}}
\put(-37,78){\shortstack{$x$}}
\put(-25,-40){\shortstack{$t=t_0$}}
\put(175,80){\shortstack{$\gamma_1$}}

\put(170,55){\circle*{3}}
\put(142,25){\circle*{3}}
\put(108,5){\circle*{3}}
\put(60,-10){\circle*{3}}

\put(-10,77){\shortstack{$\gamma_2$}}
\put(-10,50){\shortstack{$\gamma_2$}}
\put(120,49){\shortstack{$\gamma_2$}}

\put(160,34){\shortstack{$B$}}
\put(126,6){\shortstack{$A$}}
\put(85,-12){\shortstack{$B$}}

\put(60,9){\shortstack{$\gamma_2$}}
%\put(-10,40){\shortstack{$\gamma_2$}}
\put(103,25){\shortstack{$\gamma_2$}}

\put(-10,30){\shortstack{$\gamma_2$}}
\put(-10,1){\shortstack{$\gamma_2$}}

\put(60,-50){\shortstack{Fig. 6}}

\end{picture}

On the other hand, it follows from the first claim of the lemma that
on each interval of the interior of $\textsf{B}$, $(r_2)_x\geq 0$
and since $r_2$ is preserved along characteristics of the second
family $r_2$ is an increasing function on each interval of the
interior of $\textsf{B}$ and therefore $-u$ is increasing also due
to (\ref{q}).

Summing up these two properties we get for the function $-u$ along
$\gamma_1$ the following possibilities: either the set $\textsf{A}$
is bounded and then $-u$ is increasing function along $\gamma_1$
starting from a certain point, or the set $\textsf{A}$ is unbounded
but then $-u$ must be bounded on the whole $\gamma_1$, since $-u$ is
uniformly bounded on the intervals of $\textsf{A}$ which may
alternate with intervals of $\textsf{B}$ where the function $-u$ is
monotonic. Notice that the last possibility cannot happen in fact
since we are in the case of convergent integral (\ref{int}). So we
have that $-u$ is a monotonic increasing function and the claim
follows. Lemma is proved.
\end{proof}

Lemma \ref{l1} enables us to distinguish between two types of
characteristics which start at $t_0$ in a positive or negative
direction of time as follows.

\begin{definition}

Let $\gamma$ be a characteristic curve lying in the Hyperbolic
domain $U_\alpha$ defined on a \textit{maximal} interval $[t_0,t_+)$
(or respectively $(t_-,t_0]$). We shall say that $\gamma$ is of type
$B_+$ (res. $B_-$) if $t_+=+\infty$ (resp. $t_-=-\infty$) and
$-u\rightarrow +\infty$ when $t\rightarrow +\infty$ (resp.
$t\rightarrow -\infty$).

We shall say that $\gamma$ is of type $A_+$ (resp. $A_-$) in the
opposite case. That is if either $t_+$ (resp. $t_-$) is finite, or
$t_+=+\infty$ (resp. $t_-=-\infty$) and $-u$ does not tend to
$+\infty$ when $t\rightarrow +\infty$ (resp. $t\rightarrow
-\infty$).
\end{definition}

By Lemma \ref{l1} if $\gamma$ is of type $A_+$ then $(r)_x\leq 0$ along $\gamma$, and
if $\gamma$ is of type $A_-$ then $(r)_x\geq 0$ along $\gamma$.

\begin{lemma}
\label{semi-infinite}
 There cannot exist two semi-infinite
characteristics in the same direction
$$\gamma_1=(t,x_1(t)),\ \gamma_2=(t, x_2(t))$$ of the first and the
second eigenvalue such that both of them belong to the same class
$B_{\pm}$.

\end{lemma}
\begin{proof}

Assume on the contrary that there exist such $\gamma_1=(t,x_1(t)),\
\gamma_2=(t, x_2(t))$ belonging to the same class, say $B_+$, so
that $-u|_{\gamma_{1}}\rightarrow +\infty$ and $-u|_{\gamma_{2}}\rightarrow +\infty$ when
$t\rightarrow+\infty$.

Then by periodicity we can shift the characteristics to get
$$\gamma_1^{(k)}=(t,x_1(t)+k),\  \gamma_2^{(l)}=(t,x_2(t)+l)$$ which are characteristics
of class $B_+$ also for all $k,l \in \mathbf{Z}$. Since the
functions $x_1,x_2$ are solutions of the ODEs $$\dot{x}=\pm
\sqrt{-\sigma'(u)}$$ respectively, it follows that $x_1$
(respectively $x_2$) are strictly monotone increasing (respectively
decreasing) function with the derivative bounded away from zero.
Therefore for sufficiently large $k$ the characteristics $\gamma_1$
and $\gamma_2^{(k)}$  must intersect in a unique point, call it
$P_k$ (see  Fig. 7).

\vskip2.5cm

\begin{picture}(170,100)(-100,-80)
\put(-35,-25){\vector(0,4){100}}
\put(-35,-25){\vector(4,0){220}}
\qbezier(-20,10)(75,15)(150,60)
\qbezier(-20,60)(80,56)(150,10)
\qbezier(-20,35)(80,31)(150,-15)

\put(115,33){\shortstack{$P_k$}}
\put(81,20){\shortstack{$P_{k-1}$}}

\put(188,-28){\shortstack{$t$}}
\put(-37,78){\shortstack{$x$}}
\put(-33,10){\shortstack{$\gamma_1$}}
\put(-33,60){\shortstack{$\gamma_2^{(k)}$}}
\put(-33,35){\shortstack{$\gamma_2^{(k-1)}$}}
\put(101,36){\circle*{3}}
\put(64,23){\circle*{3}}

\put(60,-50){\shortstack{Fig. 7}}

\end{picture}

Denote by $t_k$ the $t$-coordinates of $P_k$. One can see that $t_k$
is monotone increasing and must tend to $+\infty$. Indeed in the
opposite case there exist limits $t_k\nearrow t_*$ and
$P_k\rightarrow P_*$ so that the characteristic $\gamma_2$ lies in
the half plane $t<t_*$ and tends to $-\infty$ when $t\rightarrow
t_*$. But this contradicts the fact that solutions of the equation
(\ref{ode}) are bounded on any compact interval of time.

Therefore we have,
$$-u(t_k,x(t_k))\rightarrow+\infty ,\ k\rightarrow+\infty,$$
and then by formula (\ref{q}) also
$$ r_2(P_k)-r_1(P_k)\rightarrow+\infty ,\ k\rightarrow+\infty.$$

 But this is not possible since by periodicity in $x$ of $(u,v)$ one
 has that
 $r_1(t_0,x)$ and $r_2(t_0,x)$ are bounded, and so by conservation of $r_1,r_2$
 along characteristics $
 r_2(P_k)-r_1(P_k)$ must be bounded also. This contradiction proves the lemma.
\end{proof}

\section{Hyperbolic part of the proof of main Theorems \ref{main1}, \ref{main2}}

Recall that Lemma \ref{semi-infinite} means that if at least one of
the characteristics of the first eigenvalue $\gamma_1$ is unbounded
on $[t_0,+\infty)$ with the property that $-u\rightarrow +\infty$
along it, then along any characteristic of the second eigenvalue one
has: $(r_2)_x<0$.

Let us prove now Theorem \ref {main1}.
\begin{proof}
Assume without loss of generality that the initial data $u(t_0,x)$
lies within the Hyperbolic domain $U_\alpha$. Introduce
$$t'= \sup \{t: [t_0,t]\times \mathbf{S}^1 \subseteq U_\alpha\},
$$ this means that $t'$ is the first moment where non-Hyperbolic
type appears. In other words $u$ becomes equal to $\alpha$ at some
point on the circle $\{t'\} \times \mathbf{S}^1.$  Write
$U'=[t_0,t')\times \mathbf{S}^1$. We prove that $t'$ equals in fact
to $+\infty$. Indeed,
 it follows from Lemma \ref{semi-infinite} that all characteristics of at least one of the
 eigenvalues are of class $A_+$. Without loss of generality let it be the family of the second
 eigenvalue with this property. Then it follows from Lemma \ref{l1} that
 $$(r_2)_x(t_0,x)\leq 0,
$$
holds true for every $x$. But by periodicity this is possible only
when $r_2(t_0,x)$ is in fact constant for the initial moment and so
also everywhere on the whole $U'$. This means that within the domain
$U'$ only $r_1$ can vary. But then $u$ is a function of $r_1$ only
and therefore has constant values along characteristics of the first
eigenvalue. By the construction there exists a point, say $E$, on
$\{t'\} \times \mathbf{S}^1$ where $u=\alpha$. It follows from
continuous dependence of the solutions of the ODE (\ref{ode}) on the
initial data that there exists a characteristic of the first family
terminating at the point $E$, so that $u=\alpha$ also on the whole
characteristic. But this is a contradiction, since $u<\alpha$ for
all points inside $U_\alpha$. This implies that the hyperbolic
domain $U_\alpha$ coincides with the whole semi-infinite cylinder $
[t_0,+\infty)\times \mathbf{S}^1 $.

Furthermore, since we  know that $r_2$ is a constant on the whole
half cylinder then $u$ depends only on $r_1$ and has constant values
along characteristics of the first family (in particular $-u$ does
not tend to infinity) so Lemma \ref{l1} implies
$$(r_1)_x\leq 0.$$ Using periodicity again we conclude that $r_1$ is constant
also everywhere on the half-cylinder. Thus $(u,v)$ is a constant
solution on the semi-infinite cylinder.
 We are done.
\end{proof}

For Theorem \ref{main2} we have to consider the whole infinite
cylinder and characteristics which may be infinite in both
directions. Also in this case elliptic domains cannot be excluded as
before, to treat them one needs an additional tool. In the next
section we treat the Elliptic region. These two steps provide the
proof of Theorem \ref{main2}. The first step goes as follows:

\begin{theorem}
Let $(u,v)$ be a $C^2$-solution of the system (\ref{2}) on the whole
infinite cylinder. Then either $U_\alpha$ or $U_\beta$ coincide with
the whole cylinder and $u,v$ are constants everywhere, or both
$U_\alpha$ and $U_\beta$ are empty, i.e. $u \in [\alpha,\beta]$
everywhere.
\end{theorem}
\begin{proof}To give a proof assume with no loss of
generality that $U_\alpha$ does not coincide with the whole cylinder
otherwise Theorem \ref{main1} yields the result.

We show that then $U_\alpha$ must be empty. We prove this by
contradiction. Fix a connected component of $U_\alpha$, denote it
$U'$, and take any initial moment $t_0$ with the property that the
intersection of $\{t=t_0\}$ with component $U'$ is not empty. It
consists of the disjoint union of open intervals --- we call them
intervals of Hyperbolicity (the case when it is the whole circle is
covered already by Theorem \ref{main1}).

Consider the following two complementary cases
%(these cases are complementary by Lemma \ref{semi-infinite}):

Case 1. \textit{For any initial moment $t_0$ with the property
$\{t=t_0\}\cap U'\neq \emptyset$, for at least one of the
eigenvalues (say for the second one) all characteristics started
from $t_0$ in positive and negative direction belong to the classes
$A_+,A_-$. }

In this case it follows from Lemma \ref{l1} that $r_2(t_0,x)$ is
constant on any interval of the intersection $\{t=t_0\}\cap U'$
(since $(r_2)_x\geq 0$ and in the same time $(r_2)_x\leq 0$). Since
$t_0$ is arbitrary and $U'$ is connected, this implies that $r_2$ is
constant on the whole connected component $U'$. This implies that
only $r_1$ varies on $U'$ and so $u,v, \lambda_1,\lambda_2$ are
functions of $r_1$ only. This means in particular that $u$ keeps
constant values along characteristics of the first eigenvalue.
Therefore every such characteristic can be extended infinitely in
both directions because if it reaches the boundary of the Hyperbolic
region $U_\alpha$, then $u$ must have value $\alpha$ on the whole
characteristic, which contradicts Hyperbolicity. Moreover, since
$\lambda_1$ is a function of $r_1$ only, so it is a constant along
$\lambda_1$-characteristics, then these characteristics are
necessarily parallel straight lines of the slope $\lambda_1$. So
$U'$ is an infinite strip of the slope $\lambda_1$. Furthermore on
the boundary $u=\alpha$ so $\lambda_1=\sigma'|_{u=\alpha}=0$, so
that $\lambda_1=0$ everywhere on $U'$. Thus these strips are
horizontal and $u$ equals $\alpha$ identically on $U'$. This
contradiction finishes the proof.

%%%%%%%%%%
Case 2. \textit{There exists $t_0$ such that $\{t=t_0\}\cap U'\neq
\emptyset$ and for each eigenvalue $\lambda_1$ and $ \lambda_2$
there exists a characteristic of class $B_{\pm}$ started at
$\{t=t_0\}\cap U'$ in some direction.}

It follows from the Lemma \ref{semi-infinite} that the directions of these two
characteristics  must be opposite. So assume without loss of
generality that $\gamma_1,\ \gamma_2$ are $\lambda_1,\
\lambda_2$-characteristics in the classes $B_+,\ B_-$ respectively.
Then it follows from the lemmas that the characteristics $\gamma_1,\
\gamma_2$ being extended beyond $t_0$ in the negative and positive
direction respectively belong to classes $A_-, A_+$ respectively.
And thus by the Lemma \ref{l1}
$$
 (r_1)_x(t_0,x)\geq 0,\ (r_2)_x(t_0,x)\leq 0,
$$
for all $x$ in the intervals of Hyperbolicity. So in this case
$(r_1-r_2)(t_0,x)$ is a monotone function in $x$. Then also
$u(t_0,x)$ is monotone  by the formula (\ref{q}) and since $u$
equals $\alpha$ at the ends of the intervals of Hyperbolicity, then
$u=\alpha$ on the whole interval of Hyperbolicity, contradiction.
This contradiction completes the proof of the theorem.
\end{proof}
Let us complete this section with the proof of Theorem
\ref{bounded}.
\begin{proof}[Proof of Theorem \ref{bounded}]
It is now very easy. The assumption that $u$ is bounded implies that
in the Hyperbolic region all characteristics belong to class $A$.
Therefore we can conclude exactly as in the Case 1 of the previous
theorem that either one of the Hyperbolic domains $U_\alpha,U_\beta$
coincides with the whole cylinder and the solution $(u,v)$ is
constant everywhere, or $U_\alpha,U_\beta$ are empty and the
solution satisfies everywhere $u\in [\alpha, \beta]$. This case is
treated in the next section where periodicity in $x$ only of the
function $u$ is needed.
\end{proof}

\section{Elliptic case}
In this section we treat the Elliptic region as follows.
\begin{theorem}
Suppose $(u,v)$ is a $C^2$-solution of the system (\ref{2}) on the
whole cylinder which satisfies $u\in[\alpha,\beta]$ everywhere, and
such that $u$ is periodic in $x$. Then $u,v$ are constants.
\end{theorem}

\begin{proof}
Take any function $f:[\alpha,\beta]\rightarrow\mathbf{R}$ satisfying
$$f(\alpha)=f(\beta)=0,$$
$$f(u)>0 \ for \ all\ u\in(\alpha,\beta)\  and$$
$$f''(u)<0 \ for \ all\  u \in [\alpha,\beta].$$

Introduce
$$
 E(t)=\int_{S^1}f(u(t,x))dx,
$$
which by the construction is a positive function of $t\in\mathbf{R}$
unless $u$ equals identically $\alpha$ or $\beta$. Compute the
second derivative of $E$ using the system (\ref{2}) and integration
by parts. Notice that for the integration by parts periodicity of
$u$ only is essential and not of $v$. We have
$$
 \ddot{E}=\int_{S_1}f''(u)\left((v_x)^2+\sigma'(u)(u_x)^2\right)dx\leq 0.
$$
Since $u\in [\alpha,\beta]$ then $\sigma'(u)$ is non-negative, by
the assumptions on $\sigma$. So we get that $E$ is a positive
concave function and thus must be constant. Then obviously $u,v$ are
constants everywhere.
\end{proof}

\end{document}